\documentclass[11pt]{article}

\usepackage{hyperref}
\usepackage{enumerate, amsthm}
\usepackage{mathrsfs}
\usepackage{rotating}
\usepackage{cancel}
\usepackage{xcolor}
\usepackage[margin= 0.8 in]{geometry}
\usepackage[margin={1.5cm, 1.5cm},font=small, labelsep=endash]{caption}
\usepackage{mathtools}
\usepackage{ifthen}

      \usepackage{amssymb}
      \usepackage{amsmath}
      \newcommand\myeq{\mathrel{\stackrel{\makebox[0pt]{\mbox{\normalfont\tiny d}}}{=}}}
      \numberwithin{equation}{section}
      \theoremstyle{plain}
      \newtheorem{theorem}{Theorem}[section]
      \newtheorem{lemma}[theorem]{Lemma}
      \newtheorem{corollary}[theorem]{Corollary}
      \newtheorem{proposition}[theorem]{Proposition}

      \theoremstyle{definition}

      \theoremstyle{remark}
      \newtheorem{remark}[theorem]{Remark}

      \newcommand{\R}{{\mathbb R}}
      \newcommand{\E}{\mathbb E}

      \renewcommand{\P}{\mathbb P}
      
      \newcommand{\var}{\mathrm{Var}}

      \newcommand{\F}{\mathbb{F}}

      \renewcommand{\H}{\mathcal H}
      \makeatletter
      \def\@setcopyright{}
      \def\serieslogo@{}
      \makeatother
      
\begin{document}
\title{Finite size scaling of random XORSAT}

\author{Subhajit Goswami\\
University of Chicago
}
\date{}

\maketitle

\begin{abstract}
We consider a ``configuration model'' for random XORSAT which is a random system of $n$ equations over $m$ variables in $\F_2$. Each equation is of the form $y_1 + y_2 + \cdots + y_k = b$ where $k \geq 3$ is fixed, $y_1, y_2, \cdots$ are variables (not necessarily 
distinct) and $b \in \F_2$. The equations are chosen independently and uniformly at random with replacement. 
It is known \cite{Dubois02, Dietzfelbinger10, pittel2016} that there exists $\rho_k$ such that $m / n = \rho_k$ is a sharp threshold 
for the satisfiability of this system. In this note we show that for the configuration model, the width of SAT-UNSAT transition window for random $k$-XORSAT is $\Theta(n^{-1/2})$ and also derive the exact scaling function.\newline

\smallskip
\noindent{\bf Key words and phrases.} Random $k$-XORSAT, Random constraint satisfaction problems, Finite size scaling.
\end{abstract}

\section{Introduction}
Consider a random instance of a system of $n$ equations over $m$ variables in $\F_2$ as 
follows. The $j$-th equation ($j \in [n]$) is of the form $y_1^j + y_2^j + \cdots + y_k^j = b^j$, where $y_1^j, y_2^j, \cdots, y_k^j$ are chosen independently and uniformly from the variables $x_1, x_2, \cdots, x_m$; $b_j$ is uniform in $\{0, 1\}$ and is 
independent with $y_1^j, y_2^j, \cdots, y_k^j$. The tuples $(y_1^j, y_2^j, \cdots, y_k^j; b_j)$'s are also 
independent. We refer to this system as $\mathcal E_{k, 
m, n}$. A natural object of interest related to $\mathcal E_{k, m, n}$ is:
$$P_k(m, n) = \P(\mathcal E_{k, m, n}\mbox{ is solvable in }\F_2)\,.$$
It is known that $P_k(m, n)$ exhibits a sharp phase transition around a critical value $\rho_k$ of the ratio 
$m/n$ when $k \geq 3$. More precisely, $\lim_{m, n \to \infty; m/n \to \rho}P_k(m, n) = 1$ or 0 accordingly as 
$\rho > \rho_k$ or $< \rho_k$ respectively. This was first shown by Dubois and Mandler \cite{Dubois02} for $k = 3$ and independently by Dietzfelbinger et. al. \cite{Dietzfelbinger10} and Pittel and Sorkin \cite{pittel2016} for all $k \geq 3$. In this paper we determine the finite size scaling behavior of $P_k(m, n)$ around the threshold 
$\rho_k$. Our main result is the following theorem.
\begin{theorem}
\label{thm-main}
Let $k \geq3$ and $m = \lfloor n\rho_k + rn^{1/2}\rfloor$ 
for some $r \in \R$. There exist positive numbers 
$s_k, C^*_k$ depending only on $k$ and a positive constant $c^*$ such that for all large enough $n$
\begin{equation*}
\label{eq:thm-main}
\big|P_k(m, n) - \Phi(rs_k)\big| \leq C^*_kn^{-c^*}\,,
\end{equation*}
where $\Phi(.)$ is the standard Gaussian distribution function.
\end{theorem}
\subsection{Backgrounds and related works}
The existence of a sharp threshold for satisfiability of random $k$-XORSAT for general $k \geq 3$ was established separately by Dietzfelbinger et. al. \cite{Dietzfelbinger10} and Pittel 
and Sorkin \cite{pittel2016}. Their approaches are somewhat different and for the purpose of this paper we 
will discuss the latter work in particular. In their paper Pittel and Sorkin first derived a similar threshold for what they called a ``constrained'' $k$-XORSAT model (introduced by Dubois and Mandler in \cite{Dubois02}) where the system of equations is uniformly random over the subclass of systems in which each variable appears at 
least twice. They found the critical ratio of the number of variables to number of equations for this constrained 
setup as 1. The threshold $\rho_k$ for unconstrained model was then derived as the variables-equations ratio such that the same in the \emph{core} of the associated hypergraph approaches 1 in probability as $n$ becomes 
large (see section~\ref{sec:sketch}, first paragraph for a detailed 
discussion). They refined their result for the constrained setup (\cite[Theorem~2]{pittel2016}) to include cases when $m - n$ wanders 
off to $\infty$ or $-\infty$ arbitrarily slowly. This sharp transition implies the scaling factor for the transition window should be $n^{-1/\nu}$ if the typical fluctuation in the difference between number of variables and equations 
in the core is $\Theta(n^{1 - 1/\nu})$ for some $\nu > 1$. 
As a candidate for the scaling exponent $\nu$, numerical simulations \cite{leone2001} seem to suggest the value 
$\nu = 2$ for $k = 3$. This is also the lower bound (assuming existence) proved in \cite{Wilson02} for 
a class of problems including random XORSAT. Among other problems in the class perhaps the most relevant for us is the appearance of a non-empty core in random $k$-uniform 
hypergraphs. For this problem the lower bound was indeed found to be the true scaling exponent (see \cite{AMRU09, 
DM08}). A refined scaling law was established by Dembo and Montanari in \cite{DM08} where they analyzed a naive 
algorithm for obtaining the core. In the same paper they remarked (see \cite[Remark~2.6]{DM08}) that their techniques were applicable to a wide variety of properties of the core ``in the scaling regime'' $\rho = \rho_c + rn^{-1/2}$ where $\rho_c$ is the threshold for 
the appearance of core. In this note we use the tools developed in \cite{DM08} to determine the asymptotic distribution of the difference between number of variables and equations in the core (see Proposition~\ref{prop:prep}) in a \textbf{different regime} namely when $\rho_n$ approaches the 
satisfiability threshold $\rho_k$. Combined with the result of \cite{pittel2016} this then yields the scaling law of random $k$-XORSAT.
\begin{remark}
\label{remark:1} The usual random model for $k$-XORSAT is where each of the $n$-equations' $k$ variables are drawn uniformly \emph{without replacement} from the set 
of all $m$ variables. In contrast we work with a \emph{configuration model} in this paper which allows 
same variable to appear multiple times in an equation. A similar model was considered in \cite{pittel2016} for 
the constrained set up (see section~3). In fact the bounds in \cite[Theorem~2]{pittel2016} were first derived for this configuration model which has the critical threshold 1 and the uniformly random set up was subsequently tackled using them (see Lemma~7 and 
Corollary~8). As a result the critical threshold for our 
model is same as $\rho_k$. The main difficulty with the uniformly random model is the analysis of the core and is purely technical in nature (see 
\cite[Remark~2.7]{DM08} for more discussions on this). 
However we expect the finite size scaling behavior to be identical for these two models up to a renormalization of the scaling function.
\end{remark}

\subsection{A word on the organization}
In section~\ref{sec:sketch}, we discuss the results from \cite{DM08} and \cite{pittel2016} that are relevant 
for the current paper and provide a sketch of our proof. In section~\ref{sec:ODE} we discuss the properties of some ODE's which describe the evolution of a Markov process associated with the systems of equations 
\emph{leading} to the core. Finally in section~\ref{sec:approximation}, we derive the asymptotic distribution of the difference between number of variables and equations in the core.

\medskip

\noindent {\bf Acknowledgement.} We thank Jian Ding, Amir Dembo, Yash Kanoria and Steve Lalley for helpful discussions. We especially thank Jian Ding for suggesting the problem and carefully reviewing an early manuscript of the paper.

\section{The discussion of previous results and the sketch of proof}
\label{sec:sketch}
Central to our analysis is the so called \emph{peeling} algorithm which removes equations from the system one at a time as long as there is a variable appearing exactly once. Before we make it precise, let us link the system to a 
$k$-uniform \emph{directed} hypergraph with $m$ vertices and $n$ hyperedges. We do this by identifying the variables $x_1, x_2, \cdots, x_m$ with vertices $v_{x_1}, v_{x_2}, \cdots, 
v_{x_m}$ of a hypergraph $\H_{k, m, n}$ and then including, for each equation $y_1^j + y_2^j + \cdots + y_k^j = b^j$ in the system, the \emph{ordered} list of vertices $\big(v_{y_1}, 
v_{y_1}, \cdots, v_{y_k}\big)$ as a hyperedge. The degree of a vertex in $\H_{k, m, n}$ (or the variable it corresponds to) is the total number of times it occurs in all the hyperedges of 
$\H_{k, m, n}$ counting repetitions. The \emph{2-core} (or simply the \emph{core}) of  $\H_{k, m, n}$ is the maximal subhypergraph such that minimum degree of vertices in it is at least 2. 
In terms of $\mathcal E_{k, m, n}$ it is the largest subsystem of equations such that any variable appearing in it appears at least 
twice. We will also refer to this subsystem as the core and the particular usage should be clear from the context.

At each step the peeling algorithm removes an equation from the system which it picks 
uniformly from the set of all equations containing at least one variable of degree 1. It is 
easy to see that this algorithm stops at the core. We can naturally associate a $\mathbb{Z}^2$-valued process $\{\overrightarrow{z}(\tau)=(z_1(\tau),z_2(\tau)), n \geq \tau \geq 0\}$ with this algorithm, where $z_1(\tau)$ and $z_2(\tau)$ are respectively the number of variables of degree 1 and $\geq 2$ after $\tau$ steps. As proved in \cite[Lemma~3.1]{DM08}, this process is a time inhomogeneous Markov Process. Let us denote its transition probabilities as $W_\tau^+(\Delta\overrightarrow{z}|\overrightarrow{z}) = \mathbb{P}(\overrightarrow{z}(\tau + 1) = \overrightarrow{z} + \Delta\overrightarrow{z}|\overrightarrow{z}(\tau) = \overrightarrow{z})$. For large $n$ and ``suitable'' range of $\tau$ this transition kernel is well approximated (\cite[Lemma 4.5]{DM08}) by a simpler transition kernel given by,
\begin{eqnarray}
\label{eq:2.1}
\widehat \P(\overrightarrow{z}(\tau + 1) = \overrightarrow{z} + (q_1 - q_0, -q_1)|\overrightarrow{z}(\tau) = \overrightarrow{z})
= \binom{k-1}{q_0 - 1, q_1, q_2}\mathfrak{p}_0^{q_0 - 1}\mathfrak{p}_1^{q_1}\mathfrak{p}_2^{q_2},
\end{eqnarray}
where $q_0 + q_1 + q_2 = k$ and $\mathfrak p_0, \mathfrak p_1, \mathfrak p_2$ are defined as follows. For $\overrightarrow{x} = \overrightarrow{z}/n, \theta = \tau/n$,
\begin{equation}
\label{eq:2.2}
\mathfrak{p}_0 = \frac{\mbox{max}(x_1, 0)}{k(1-\theta)}, \mathfrak{p}_1 = \frac{x_2\lambda^2}{k(1-\theta)(e^\lambda - 1 - \lambda)}, \mathfrak{p}_2 = \frac{x_2\lambda}{k(1-\theta)},
\end{equation}
where for $x_2 > 0$, $\lambda$ is the unique positive solution of
\begin{equation}
\label{eq:2.3}
f_1(\lambda) \equiv \frac{\lambda(e^\lambda - 1)}{e^\lambda - 1 - \lambda} = \frac{k(1-\theta)- \max(x_1,0)}{x_2}
\end{equation}
so that $\mathfrak{p}_0 + \mathfrak{p}_1 + \mathfrak{p}_2 = 1$ whereas for $x_2 = 0$, $\mathfrak{p}_1$ and $\mathfrak{p}_2$ are set at $0$ and $1 -\mathfrak{p}_0$ respectively by continuity. We denote this transition kernel by $\widehat{W}_{\theta}(\Delta\overrightarrow{z}|\overrightarrow{x})$. In fact the definition can be extended to include all vectors in $\overrightarrow{x} \in \mathbb{R}^2$ by defining
$$\widehat{W}_{\theta}(\Delta\overrightarrow{z}|\overrightarrow{x}) = \widehat{\mathbb{P}}(\overrightarrow{z}(\tau + 1) = \overrightarrow{z} + (q_1 - q_0, -q_1)|\overrightarrow{z}(\tau)/n = K_{\theta}(\overrightarrow{x})),$$
where for each $\theta\in[0, 1)$ $K_{\theta}:\mathbb{R}^2 \rightarrow \mathcal{K}_{\theta}$ denotes the projection onto the convex set $\mathcal{K}_{\theta}\equiv\{\overrightarrow{x}\in \mathbb{R}_+^2:x_1 + 2x_2\leq k(1 - 
\theta)\}$. A consequence of this approximation is that we can couple the Markov processes with kernels $W_\tau^+(.|\overrightarrow{z})$ and $\widehat{W}_{\tau/n}(\Delta\overrightarrow{z}|\overrightarrow{z}/n)$ (and the same initial 
conditions) while keeping their ``graphs'' close to each other with high probability. In order to state this result precisely we need to introduce some notations and definitions. Let $\mathcal{G}_k(n,m)$ denote the collection of all possible instances of $\H_{k, m, n}$ and $\mathbb{P}_{\mathcal{G}_k(n,m)}(.)$ denote the uniform distribution on the set. Define two $\mathbb{Z}^2$-valued Markov chains with distributions $\mathbb{P}_{n,\rho}(.)$ and $\widehat{\mathbb{P}}_{n,\rho}(.)$ respectively as follows:$$\mathbb{P}_{n,\rho}(\overrightarrow{z}(0)=\overrightarrow{z}) = \widehat{\mathbb{P}}_{n,\rho}(\overrightarrow{z}(0)=\overrightarrow{z})=\mathbb{P}_{\mathcal{G}_k(n,m)}(\overrightarrow{z}(G)=\overrightarrow{z}),$$
if $\overrightarrow{z}\in\mathbb{Z}_+^2$ such that $z_1 + 2z_2 \leq nk$, and $0$ otherwise. Coming to the transition kernels define
$$W_{\tau}(\Delta\overrightarrow{z}|\overrightarrow{z}) = \begin{cases} 
W_{\tau}^+(\Delta\overrightarrow{z}|\overrightarrow{z}) & \mbox{if } z_1\geq1,n^{-1}\overrightarrow{z}\in\mathcal{K}_{\tau/n},\\ 
\widehat{W}_{\tau/n}(\Delta\overrightarrow{z}|n^{-1}\overrightarrow{z}) & \mbox{otherwise}\,.
\end{cases}$$
Now evolve the two Markov chains according to:
\begin{eqnarray*}
\mathbb{P}_{n,\rho}\big(\overrightarrow{z}(\tau + 1) = \overrightarrow{z} + \Delta\overrightarrow{z}|\overrightarrow{z}(\tau) = \overrightarrow{z}\big) = W_{\tau}(\Delta\overrightarrow{z}|\overrightarrow{z})\,, \mbox{ and}\\
\widehat{\mathbb{P}}_{n,\rho}\big(\overrightarrow{z}(\tau + 1) = \overrightarrow{z} + \Delta\overrightarrow{z}|\overrightarrow{z}(\tau) = \overrightarrow{z}\big) = \widehat{W}_{\tau/n}(\Delta\overrightarrow{z}|n^{-1}\overrightarrow{z})
\end{eqnarray*}
for $\tau = 0,1,\ldots,n-1$. Our original Markov process is not quite same as the one 
associated with $\mathbb{P}_{n,\rho}(.)$. However, they coincide until the first time $\tau$ such that $z_1(\tau) = 0$, i.e., when the peeling algorithm terminates at the core. For this reason whenever we mention the process $\{\overrightarrow{z}(\tau)\}$ without any reference to its distribution, the latter 
is implicitly understood to be $\mathbb{P}_{n,\rho}$. Now we can state the lemma we were looking for.
\begin{lemma}\cite[Lemma~5.1]{DM08}
\label{lem:coupling}
There exist finite $C_* = C_*(k,\epsilon)$ and positive $\lambda_* = \lambda_*(k, \epsilon)$, and a coupling between $\{\overrightarrow{z}(\tau)\}$
and $\{\overrightarrow{z}'(\tau)\}\myeq\widehat{\mathbb{P}}_{n,\rho}(.)$, such that for any $n$, $\rho \in [\epsilon, 1/\epsilon]$ and $r > 0$,
\begin{equation*}
\label{eq:approx_kernel4}
\mathbb{P}\big(\sup\limits_{\tau \leq \tau_*}||\overrightarrow{z}(\tau) - \overrightarrow{z}'(\tau)||>r\big) \leq C_*e^{-\lambda_*r}\,,
\end{equation*}
where $\tau_* \leq n$ denotes the first time such that $(\overrightarrow{z}(\tau_*),\tau_*) \notin \mathcal{Q}(\epsilon)$ and for each $\epsilon > 0$,
\begin{eqnarray*}
\mathcal{Q}(\epsilon)\equiv\{(\overrightarrow{z},\tau):-nk + n\epsilon \leq z_1; n\epsilon<z_2; 0 \leq \tau \leq n(1 - \epsilon); n\epsilon\leq(n - \tau)k - \max(z_1,0) - 2z_2\}\,.
\end{eqnarray*}
\end{lemma}
Another important ingredient is the asymptotic joint distribution of number of degree 1 and 2 vertices 
in $\H_{k, m, n}$. For $\overrightarrow{\mu}\in \mathbb{R}^d$ and a positive definite $d$-dimensional matrix $\Sigma$, denote the $d$-dimensional Gaussian density of mean $\overrightarrow{\mu}$ and covariance $\Sigma$ by $\phi_d(.|\overrightarrow{\mu};\Sigma)$. Entries of the vector $\overrightarrow{z} = (z_1, z_2)$ denote the number of vertices of degree 1 and at least 2 respectively in a random graph drawn uniformly from $\mathcal{G}_k(n,\lfloor n\rho\rfloor)$. Define
\begin{equation}
\label{eq:intial_cond_vec}
\overrightarrow{y} = \overrightarrow{y}(\rho) = (k e^{-k/\rho},\rho(1 - e^{-k/\rho})-k e^{-k/\rho})\,,
\end{equation}
and denote by $\mathbb{Q} = \mathbb{Q}(\rho)$ some positive definite matrix which we will specify later. Then,
\begin{lemma}\cite[Lemma~4.4]{DM08}
\label{lem:normal_approx}
For any $\epsilon >0 $ there exist positive 
constants $\kappa_0, \kappa_1, \kappa_2, \kappa_3$, such that for all $n$, $r$, and $\rho\in[\epsilon, 1/\epsilon]$,
$$||\mathbb{E}\overrightarrow{z} - n\overrightarrow{y}|| \leq \kappa_0\,,$$
$$\mathbb{P}(||\overrightarrow{z} - \mathbb{E}\overrightarrow{z}||\geq r) \leq \kappa_1e^{-r^2/\kappa_2n}\,,$$
and
$$\sup\limits_{\overrightarrow{u}\in \mathbb{R}^2}\sup\limits_{x\in \mathbb{R}}\bigl\lvert\mathbb{P}(\overrightarrow{u}.\overrightarrow{z}\leq x)-\int\limits_{\overrightarrow{u}.\overrightarrow{z}\leq x}\phi_2(\overrightarrow{z}|n\overrightarrow{y};n\mathbb{Q})d\overrightarrow{z}\bigr\rvert \leq \kappa_3n^{-1/2}\,.$$
\end{lemma}
One immediate consequence of Lemma~\ref{lem:normal_approx} is that 
$n^{-1}\overrightarrow{z}(0)$ converges in probability to $\overrightarrow{y}(0)$. A natural question then is if we can say something similar about $n^{-1}\overrightarrow{z}(n\theta)$. 
If the convergence still holds, one and perhaps the only reasonable candidate for the limit would be $\lim\limits_n n^{-1}\mathbb{E}\overrightarrow{z}(n\theta) = 
\overrightarrow{y}(\theta)$ (say). Since the transition kernel in \eqref{eq:2.1} depends on $\overrightarrow{z}$ and $\tau$ only through the scaled variables $\overrightarrow{x}$ and $\theta$, the gradient of $\overrightarrow{y}(\theta)$ should roughly equal $n^{-1}\mathbb{E}(\Delta \overrightarrow{z}(n\theta))/ n^{-1} = \mathbb{E}(\Delta \overrightarrow{z}(n\theta))$ where the expectations are with respect to 
the same kernel. We further get from \eqref{eq:2.1} that this expectation is $(-1 + (k-1)(\mathfrak{p}_1 - \mathfrak{p}_0),-(k-1)\mathfrak{p}_1)$ which we denote by 
$\overrightarrow{F}\big(\overrightarrow{y}(\theta),\theta \big)$. Thus the process $\{\overrightarrow{z}(n\theta)/n\}_{0 \leq \theta < 1}$ can be hoped to concentrate around the solution of the ODE,
\begin{equation}
\label{eq:ODE}
\frac{d\overrightarrow{y}}{d\theta}(\theta) = \overrightarrow{F}\big(\overrightarrow{y}(\theta),\theta \big),
\end{equation}
with the initial condition $\overrightarrow{y}(\theta) = \overrightarrow{y}$ from 
\eqref{eq:intial_cond_vec}. We denote the solution to \eqref{eq:ODE} subject to the initial conditions \eqref{eq:intial_cond_vec} as $\overrightarrow{y}(\theta,\rho)$ 
although the dependence on $\rho$ will often be suppressed. 
We discuss the analytical properties of $\overrightarrow{y}(\theta,\rho)$ in section~\ref{sec:ODE}, but for the time being let us 
precisely formulate the concentration of $\overrightarrow{z}$ around it. First we present a 
similar result for $\widehat \P_{n,\rho}$.
\begin{lemma}\cite[Lemma 5.2]{DM08}
\label{lem:concentration}
For any $k\geq3$ and $\epsilon > 0$ there exist positive 
$\eta \leq \epsilon$, and $C_0,C_1,C_2,C_3$, such that, for any $n$, $\rho \in [\epsilon, 1/\epsilon]$ and $\tau \in \{0, \ldots, \lfloor n(1-\epsilon)\rfloor\}$,
\begin{enumerate}[(a)]
\item $\overrightarrow{z}(\tau)$ is exponentially concentrated around its mean
\begin{equation*}
\label{eq:concentration1}
\widehat{\mathbb{P}}_{n,\rho}(\parallel\overrightarrow{z}(\tau) - \mathbb{E}\overrightarrow{z}(\tau)\parallel\geq r) \leq 4e^{-r^2/C_0n}.
\end{equation*}
\item $\overrightarrow{z}(\tau)$ is close to the solution of the ODE \eqref{eq:ODE},
\begin{equation*}
\label{eq:concentration2}
\mathbb{E}\parallel\overrightarrow{z}(\tau) - n\overrightarrow{y}(\tau/n)\parallel \leq C_1\sqrt{n\log{n}}.
\end{equation*}
\item$(\overrightarrow{z}(\tau),\tau) \in \mathcal{Q}(\eta)$ with high probability; more precisely,
\begin{equation*}
\label{eq:concentration3}
\widehat{\mathbb{P}}_{n,\rho}((\overrightarrow{z}(\tau),\tau) \notin \mathcal{Q}(\eta)) \leq C_2e^{-C_3n}.
\end{equation*}
\end{enumerate}
\end{lemma}
Part (c) and Lemma~\ref{lem:coupling} should give us 
analogous results for $\P_{n, \rho}$. This is confirmed by the following two corollaries.
\begin{corollary}[Corollary 5.4., \cite{DM08}]
\label{cor:5.4}
For any $\epsilon > 0$, there exists $0 < \eta < \epsilon$ and positive, finite constants $C_4, C_5$ such that if $\rho \in [\epsilon, 1/\epsilon]$, then
\begin{equation*}
\mathbb{P}_{n,\rho}((\overrightarrow{z}(\tau),\tau) \in \mathcal{Q}(\eta)\mbox{ }\forall\mbox{ }0 \leq \tau \leq n(1 - \epsilon)) \geq 1 - C_4e^{-C_5n}.
\end{equation*}
\end{corollary}
\begin{corollary}
\label{cor:5.5}
For any $\epsilon > 0$, there exist finite, positive $A = A(k,\epsilon)$, $C = C(k,\epsilon)$, and a coupling between $\{\overrightarrow{z}(\tau)\}\myeq\mathbb{P}_{n,\rho}(.)$ and $\{\overrightarrow{z}'(\tau)\}\myeq\widehat{\mathbb{P}}_{n,\rho}(.),$ such that for any $n$, $\rho \in [\epsilon,1/\epsilon]$,
\begin{equation*}
\mathbb{P}\big(\sup\limits_{\tau\leq n(1-\epsilon)}||\overrightarrow{z}(\tau) - \overrightarrow{z}'(\tau)||\geq A\log n\big) \leq Cn^{-1}\,.
\end{equation*}
\end{corollary}
\begin{proof}
This follows immediately from the previous corollary and Lemma~\ref{lem:concentration}.
\end{proof}
In section~\ref{sec:ODE} we show that there exists $\theta_k \in (0, 1)$ such that (1) $\theta_k = \min_{\theta \in [0, 1]}\{\theta \in [0, 1]: y_1(\theta, \rho_k) = 0\}$ and (2) $y_2(\theta_k, \rho_k) = 1 - \theta_k$. Furthermore in a small neighborhood of $\theta_k$, $$y_1(\theta, \rho_k) \approx \frac{\partial y_1}{\partial \theta}(\theta_k,\rho_k)(\theta - \theta_k)\,, \mbox{ and }y_2(\theta)\approx 1 - \theta_k + \frac{\partial y_2}{\partial \theta}(\theta_k,\rho_k)(\theta - \theta_k)\,,$$
where both of these partial derivatives are negative. 
Fluctuations of $\overrightarrow{z}(n\theta_k)$ around $\overrightarrow{y}(n\theta_k)$ are gained in $n\theta_k$ stochastic steps, and are therefore should be of order $\sqrt{n}$. We show in section~\ref{sec:approximation} that the rescaled variable $(\overrightarrow{z}(n\theta_k) - \overrightarrow{y}(n\theta_k))/\sqrt{n}$ converges to a Gaussian random vector. Its covariance matrix $\mathbb{Q}(\theta, \rho) = \{\mathbb{Q}_{ab}(\theta, \rho); 1 \leq a, b \leq 2\}$ is the symmetric positive definite solution of the ODE:
\begin{equation}
\label{eq:2.8}
\frac{d\mathbb{Q}(\theta)}{d\theta}=\mathbb{G}(\overrightarrow{y}(\theta),\theta) + \mathbb{A}(\overrightarrow{y}(\theta),\theta)\mathbb{Q}(\theta) + \mathbb{Q}(\theta)\mathbb{A}(\overrightarrow{y}(\theta),\theta)^T,
\end{equation}
where $\mathbb{A}(\overrightarrow{x},\theta) = \{A_{ab}(\overrightarrow{x},\theta);1\leq a,b\leq2\}$ for $A_{ab}(\overrightarrow{x},\theta) = \partial_{x_b}F_a(\overrightarrow{x},\theta)$, and $\mathbb{G}(\overrightarrow{x},\theta)$ is the covariance of $\overrightarrow{z}(\tau+1) - \overrightarrow{z}(\tau)$ w.r.t the transition kernel \eqref{eq:2.1}, i.e., the nonnegative definite symmetric matrix with entries
\begin{equation}
\label{eq:2.9}
\begin{cases}
\mathbb G_{11}(\overrightarrow{x},\theta) = (k-1)[\mathfrak{p_0} + \mathfrak{p_1} - (\mathfrak{p}_0-\mathfrak{p_1})^2],\\
\mathbb G_{12}(\overrightarrow{x},\theta) = -(k-1)[\mathfrak{p_0}\mathfrak{p_1} + \mathfrak{p_1}(1-\mathfrak{p_1})],\\
\mathbb G_{22}(\overrightarrow{x},\theta) = (k-1)\mathfrak{p_1}(1-\mathfrak{p_1}).
\end{cases}
\end{equation}
Like before we will omit $\rho$ from the notation when 
its values is fixed in a context. The positive definite initial condition $Q(0)$ for \eqref{eq:2.8} is computed on the original graph ensemble, and is given by
\begin{equation}
\label{eq:2.10}
\begin{cases}
\mathbb Q_{11}(0) = \frac{k}{\gamma}\gamma e^{-2\gamma}(e^{\gamma} -1 + \gamma - \gamma^2),\\
\mathbb Q_{12}(0) = -\frac{k}{\gamma}\gamma e^{-2\gamma}(e^{\gamma} -1 - \gamma^2),\\
\mathbb Q_{22}(0) = \frac{k}{\gamma}e^{-2\gamma}[(e^{\gamma} -1) + \gamma(e^{\gamma} -2) - \gamma^2(1+\gamma)],
\end{cases}
\end{equation}
where $\gamma = \frac{k}{\rho}$. We can now state the following proposition whose proof will be the main focus of this paper.
\begin{proposition}
\label{prop:prep}
Let $\overrightarrow{\xi}(r)$ is a bivariate Gaussian random vector of mean $\frac{\partial\overrightarrow{y}}{\partial\rho}r$ and variance $\mathbb{Q}$ (both evaluated at $\theta = 
\theta_k$ and $\rho = \rho_k$). Denote by $n_{core},m_{core}$ the number of equations and variables in the core of $\mathcal E_{k, m, n}$ and by 
$\mathbb{P}_{core}(.)$ their joint law. Then there exists a positive constant $\eta$ such that for all $A > 0$, $r \in \R$, and $n$ large enough, if $\rho_n = \rho_k + rn^{-1/2}$, we have
\begin{equation}
\label{eq:prepeq1}
|\mathbb{P}_{core}(m_{core} - n_{core}\geq A\log n) - \mathbb{P}(\xi_1(r) + \xi_2(r) \geq 0)| \leq n^{-\eta}\,,
\end{equation}
and
\begin{equation}
\label{eq:prepeq2}
|\mathbb{P}_{core}(m_{core} - n_{core}\geq -A\log n) - \mathbb{P}(\xi_1(r) + \xi_2(r) \leq 0)| \leq n^{-\eta}\,.
\end{equation}
\end{proposition}
The key to Proposition~\ref{prop:prep} is the construction of another Markov chain as in \cite[equation 2.12]{DM08} which, within the critical time window $J_n \equiv [n\theta_k - n^{\beta},n\theta_k - n^{\beta}]$ (for some $\beta \in (\frac{1}{2},1)$), serves as a good approximation for the chain with transition kernel 
\eqref{eq:2.1}. This Markov chain is evolved as:
\begin{equation}
\label{eq:2.12}
\overrightarrow{z}'(\tau + 1) = \overrightarrow{z}'(\tau) + \widetilde{\mathbb{A}}_{\tau}\big(n^{-1}\overrightarrow{z}'(\tau) - \overrightarrow{y}(\tau/n)\big) + \Delta_{\tau}\,,
\end{equation}
where $\tau_n\equiv\lfloor n\theta_k - n^{\beta}\rfloor$, $\widetilde{\mathbb{A}}_{\tau} \equiv \mathbb{I}_{\tau < \tau_n}\mathbb{A}(\overrightarrow{y}(\tau/n,\rho),\tau/n)$ and $\Delta_{\tau}$'s are independent random variables with mean $\overrightarrow{F}(\overrightarrow{y}(\tau/n,\rho),\tau/n)$ and covariance $\mathbb{G}(\overrightarrow{y}(\tau/n,\rho),\tau/n)$. 
Denote by $\overline{\mathbb{P}}_{n,\rho}(.)$ the law of the $\mathbb{R}^2$-valued Markov chain $\{\overrightarrow{z}'(\tau)\}$, where $\overrightarrow{z}'(0)$ has the uniform distribution on the graph ensemble $\mathcal{G}_l(n,m)$ for $m \equiv \lfloor n\rho \rfloor$, and
\begin{eqnarray*}
&&\overline{\mathbb{P}}_{n,\rho}(\overrightarrow{z}'(\tau + 1)=\overrightarrow{z}'(\tau) + \Delta_{\tau} + \widetilde{\mathbb{A}}_{\tau}(n^{-1}\overrightarrow{z}' - \overrightarrow{y}(\tau/n))|\overrightarrow{z}'(\tau) = \overrightarrow{z}')\\
&&=\widehat{W}_{\tau/n}(\Delta_{\tau}|\overrightarrow{y}(\tau/n)).
\end{eqnarray*}
Then we have the following proposition regarding approximation of $\widehat{\mathbb{P}}_{n,\rho}(.)$ by $\overline{\mathbb{P}}_{n,\rho}(.)$ which is a slightly modified version of \cite[Proposition~5.5]{DM08} and is central to proving Proposition~\ref{prop:prep}.
\begin{proposition}
\label{prop:5.5}
Fixing $\beta\in(\frac{1}{2},1)$ and $\beta' < \beta$, for any $\delta > \frac{1}{4}\vee(\beta - \frac{1}{2})$, there exist constants $\alpha, c$ and a coupling of the processes $\{\overrightarrow{z}(.)\}$ of distribution $\widehat{\mathbb{P}}_{n,\rho}$ and $\{\overrightarrow{z}'(.)\}$ of distribution $\overline{\mathbb{P}}_{n,\rho}(.)$ such that for all $n$ and $|\rho - \rho_k|\leq n^{\beta'-1}$,
\begin{equation}
\mathbb{P}\big(\sup\limits_{\tau\in J_n\equiv[n\theta_k - n^\beta, n\theta_k + n^\beta]}||\overrightarrow{z}(\tau) - \overrightarrow{z}'(\tau)||\geq cn^{\delta}\big) \leq \alpha n^{-1}
\end{equation}
\end{proposition}
\begin{proof}
Very similar to the proof of 
\cite[Proposition~5.5]{DM08}. All the arguments are valid for general $\rho$ bounded away from 0 and $\infty$ except for \cite[Corollary~5.3]{DM08} which corresponds to a different threshold than ours (denoted as $\rho_c$ in the paper). But this poses no problem as we have an analogue of this corollary (see Lemma~\ref{cor:5.3}) for $\rho_k$.
\end{proof}
The main usefulness of this approximation lies in the following observation. Taking
\begin{equation}
\label{eq:2.13}
\widetilde{\mathbb{B}}_{\sigma}^{\tau} \equiv \bigg(\mathbb{I} + \frac{1}{n}\widetilde{\mathbb{A}}_{\tau}\bigg)\ldots\bigg(\mathbb{I} + \frac{1}{n}\widetilde{\mathbb{A}}_{\sigma}\bigg),
\end{equation}
for integers $0 \leq \sigma \leq \tau$ (while $\widetilde{\mathbb{B}}_{\sigma}^{\tau}\equiv\mathbb{I}$ in case $\tau < \sigma$), we see that
\begin{equation}
\label{eq:2.14}
\overrightarrow{z}'(\tau) = \widetilde{\mathbb{B}}_{0}^{\tau-1}\overrightarrow{z}'(0) + \sum\limits_{\sigma = 0}^{\tau - 1}\widetilde{\mathbb{B}}_{\sigma+1}^{\tau-1}(\Delta_{\sigma} - \widetilde{\mathbb{A}}_{\sigma}\overrightarrow{y}(\tau/n,\rho))
\end{equation}
is a sum of (bounded) independent random variables, hence of approximately Gaussian distribution which prepares the ground for Proposition~\ref{prop:prep}.

We are just steps away from proving 
Theorem~\ref{thm-main}. \cite[Lemma~3.1]{DM08} tells us that conditional on $n_{core}$ and $m_{core}$ i.e. the number of equations and variables in the core, the system is uniformly distributed on all possible instances of $\mathcal E_{k, m_{core}, n_{core}}$ that has no variable 
with degree 1. Such a system of equations was called 
\emph{constrained $k$-XORSAT} in \cite{pittel2016}. Our very last ingredient is the following theorem proved from \cite{pittel2016} (see the discussions in Remark~\ref{remark:1}).
\begin{theorem}
\label{thm:Pit}
Let $Ax = b$ be a uniformly random constrained $k$-XORSAT instance with $n$ equations and $m$ variables, with $k \geq 3$ and $m, n \to \infty$ with $\liminf n/m > 2/k$. Then, for any $w(n)\to +\infty$, if $n + w(n) \leq m$ then $Ax = b$ is almost surely satisfiable, with satisfiability probability $1 - O(n^{-(k-2)} + exp(-c^\star\omega(n)))$, while if $n \geq m + \omega(n)$ then $Ax = b$ is almost surely unsatisfiable, with satisfiability probability $O(2^{-\omega(n)})$. Here $c^\star$ is a positive constant.
\end{theorem}
\emph{Proof of Theorem~\ref{thm-main}:} Theorem~\ref{thm-main} now follows immediately from Proposition~\ref{prop:prep} and Theorem~\ref{thm:Pit}.

\section{Solutions to ODE's \eqref{eq:ODE} and \eqref{eq:2.8}}
\label{sec:ODE}
In this section we will discuss the properties of the solutions to \eqref{eq:ODE} and \eqref{eq:2.8} which will 
be used repeatedly throughout our analysis. The results are based on the continuity of $(\overrightarrow{x},\theta)\shortmid\rightarrow \mathfrak{p}_a(\overrightarrow{x},\theta)$, $a=0,1,2$ on the following compact subsets of $\mathbb{R}^2\times\mathbb{R}_{+}$:
$$\widehat{q}(\epsilon)\equiv \{(\overrightarrow{x},\theta):-k\leq x_1;0\leq x_2;\theta\in[0,1-\epsilon];0\leq(1-\theta)k - \mbox{max}(x_1,0)-2x_2\}\,,$$
and $$\widehat{q}_+(\epsilon)=\widehat{q}(\epsilon)\cap\{x_1\geq0\}\,.$$
\begin{lemma}\cite[Lemma~4.1]{DM08}
\label{lem:4.1}
For any $\epsilon > 0$, the functions $(\overrightarrow{x},\theta)\shortmid\rightarrow \mathfrak{p}_a(\overrightarrow{x},\theta)$, $a=0,1,2$ are $[0, 1]$-valued, Lipschitz continuous on $\widehat{q}(\epsilon)$. Further, on $\widehat{q}_+(\epsilon)$ the functions $(\overrightarrow{x},\theta)\shortmid\rightarrow \mathfrak{p}_a(\overrightarrow{x},\theta)$ have Lipschitz continuous partial derivatives.
\end{lemma}
The next proposition provides some important properties of the solutions for general $\rho$. 
Define for $\rho > 0$, $h_{\rho,1}(u)\equiv u - 1 + \exp(-ku^{l-1}/\rho)$ and $h_{\rho,2}(u)\equiv 1 - (1 + ku^{l-1}/\rho)\exp(-ku^{k-1}/\rho)$.
\begin{proposition}\cite[Proposition~4.2]{DM08}
\label{prop:4.2}
For any $\epsilon > 0, \theta < 1 - \epsilon$, the ODE~\eqref{eq:ODE} admits a unique solution $\overrightarrow{y}$ subject to the initial conditions \eqref{eq:intial_cond_vec}, and the ODE~\eqref{eq:2.8} admits a unique, positive definite, solution $\mathbb{Q}$ subject to the initial conditions \eqref{eq:2.10}, such that:
\begin{enumerate}[(a)]
\item
For any $\epsilon > 0, \theta < 1 - \epsilon$, we have that $(\overrightarrow{y}(\theta,\rho),\theta)$ is in the interior of $\widehat{q}(\epsilon)$, with both functions $(\theta,\rho)\shortmid\rightarrow \overrightarrow{y}$ and $\theta\shortmid\rightarrow\mathbb{Q}$ Lipschitz continuous on $(\theta,\rho)\in[0, 1 -\epsilon)\times[\epsilon, 1/\epsilon]$.
\item
Let $u(\theta)\equiv(1-\theta)^{1/k}$ and $\theta_-(\rho)\equiv\mbox{inf}\{\theta\geq0 :h_{\rho,1}(u(\theta))<0\}\wedge1$. Then, for $\theta\in[0,\theta_-(\rho)]$,
\begin{equation}
\label{eq:sol4.1}
y_1(\theta,\rho) = ku(\theta)^{k-1}[u(\theta) - 1 + e^{-\gamma u(\theta)^{k-1}}],
\end{equation}
\begin{equation}
\label{eq:sol4.2}
y_2(\theta,\rho) = \frac{k}{\gamma}[1 - e^{-\gamma u(\theta)^{k-1}} - \gamma u(\theta)^{k-1}e^{-\gamma u(\theta)^{k-1}}],
\end{equation}
(where $\gamma = k/\rho$). In particular, $(\theta,\rho)\shortmid\rightarrow \overrightarrow{y}$ is infinitely continuously differentiable
and $(\theta,\rho)\shortmid\rightarrow \mathbb{Q}$ is Lipschitz continuous on $\{(\theta, \rho) : \theta\leq\mbox{min}(\theta_-(\rho),1-\epsilon),\epsilon\leq\rho\leq1/\epsilon\}$.
\end{enumerate}
\end{proposition}
Using the previous two results we can now derive several important properties of the solution to \eqref{eq:ODE} when $\rho = \rho_k$.
\begin{proposition}
\label{prop:coro_crit}
Define $\theta_*(\rho)\equiv\mbox{inf}\{\theta\geq0 :h_{\rho,1}(u(\theta))\leq0\}$. 
Then,
\begin{enumerate}[(a)]
\item $\theta_*(\rho) = 1$ for $\rho > \mbox{max}_{x>0}\frac{k(1-e^{-x})^{k-1}}{x}$. For $\rho \leq \mbox{max}_{x>0}\frac{k(1-e^{-x})^{k-1}}{x}$, $\theta_*(\rho) = 1 - (1 - e^{-\lambda_\rho})^k$ where $\lambda_\rho$ is the maximum positive solution to $\rho = \frac{k(1-e^{-\lambda})^{k-1}}{\lambda}$. Furthermore for these values of $\rho$, $y_2(\theta_*(\rho),\rho) = \frac{k}{f_1(\lambda_\rho)}(1-\theta_*(\rho))$ where $f_1$ is the same function as in \eqref{eq:2.3}.
\item \begin{enumerate}[(I)]
\item 
$\rho_k$ satisfies the equation $y_2(\theta_*(\rho_k),\rho_k) = 1 - \theta_*(\rho_k) > 0$.
\item Denoting $\theta_*(\rho_k)$ by $\theta_k$, we have for $\rho = \rho_k$, $y_1'(\theta_k) < 0$ and $y_2'(\theta_k) < 0$.
\item Although $\overrightarrow{y}(\theta)$ (for $\rho = \rho_k$) may not be twice continuously differentiable at $\theta = \theta_k$, $\overrightarrow{y}(\theta)$ is nonetheless twice continuously differentiable when considered on $[0,\theta_k]$ and $[\theta_k,\theta_k + \epsilon]$ separately for some $\epsilon > 0$.
\end{enumerate}
    \end{enumerate}
\end{proposition}
\begin{proof}
\begin{enumerate}[(a)]
\item Since $h_{\rho,1}(u(0)) > 0$ and $h_{\rho,1}(u(1)) = 0$ we have $\theta_*(\rho) \in (0,1]$ for all $\rho > 0$. Notice that 
$$h_{\rho,1}(u) > 0 \Leftrightarrow \exp\Big(-\frac{ku^{k-1}}{\rho}\Big) > 1-u\,,$$ 
for $u \in (0,1)$. Writing $-\log(1-u) = x \in (0,\infty)$ we find that,
\begin{eqnarray*}
\theta_*(\rho) = 1 \Leftrightarrow h_{\rho,1}(u) > 0\mbox{ for }u\in(0,1) \Leftrightarrow 
x > \frac{k(1-e^{-x})^{k-1}}{\rho} \Leftrightarrow \rho > \frac{k(1-e^{-x})^{k-1}}{x}\,.
\end{eqnarray*}
Hence $\theta_*(\rho) = 1$ for $\rho > \mbox{max}_{x>0}\tfrac{k(1-e^{-x})^{k-1}}{x}$. This chain of equivalences also give us that 
$$\theta_*(\rho) = 1 - (1 - e^{-\lambda_\rho})^k$$   
for $\rho \leq \mbox{max}_{x>0}\frac{k(1-e^{-x})^{k-1}}{x}$ and $\lambda_\rho$ being the 
maximum positive solution to $\rho = \tfrac{k(1-e^{-\lambda})^{k-1}}{\lambda}$. The expression for $y_2(\theta_*(\rho),\rho)$ now follows from a routine algebra:
\begin{eqnarray*}
y_2(\theta_*(\rho),\rho) &=& \rho h_{\rho,2}(u(\theta_*(\rho))) = \frac{k(1-e^{-\lambda_\rho})^k}{\lambda_\rho} - k(1-e^{-\lambda_\rho})^{k-1} + k(1-e^{-\lambda_\rho})^k\\
                            &=&  k\frac{1 - e^{-\lambda_{\rho}}(1+\lambda_{\rho})}{\lambda_{\rho}(1 - e^{-\lambda_\rho})}(1 - e^{-\lambda_\rho})^k = \frac{k}{f_1(\lambda_\rho)}(1-\theta_*(\rho))\,.
\end{eqnarray*}
\item \begin{enumerate}[(I)]
\item \cite[Theorem~16]{pittel2016} tells us that
$$\rho_k = \frac{k(1-e^{-\lambda_k})^{k-1}}{\lambda_k}\,,$$ 
where $\lambda_k$ is the unique positive solution to $f_1(\lambda_\rho) = k$. $\lambda_k > 0$ since $\lim_{\lambda \to 0+}f_1(\lambda) = 2$ and $k \geq 3$. Thus from part (a) we get $y_2(\theta_*(\rho_k), \rho_k) = 1 - \theta_*(\rho_k) > 0$.

\item From \eqref{eq:ODE} we get at $\rho = \rho_k$, 
$$y_1'(\theta_k) = -1 + (k-1)(\mathfrak{p}_1(\overrightarrow{y}(\theta_k,\rho_k),\theta_k) - \mathfrak{p}_0(\overrightarrow{y}(\theta_k,\rho_k),\theta_k))\,.$$ 
Since $y_1(\theta_k,\rho_k) = 0$, we have
\begin{eqnarray*}
y_1'(\theta_k) = (k-1)\frac{y_2(\theta_k,\rho_k)\lambda_c^2}{k(1 - \theta_k)(e^{\lambda_c} - 1 - \lambda_c)} = (k-1)\Big(\frac{\lambda_c^2}{k(e^{\lambda_c} - 1 - \lambda_c)} - \frac{1}{k-1}\Big)
\end{eqnarray*}
where $f_1(\lambda_c) = \tfrac{k(1-\theta_k)}{y_2(\theta_k,\rho_k)} = k$ i.e. $\lambda_c = 
\lambda_k$ (see the discussion around \eqref{eq:2.3}). 
Hence we can write 
$$y_1'(\theta_k) = (k-1)\Big(\frac{\lambda_l}{e^{\lambda_k} - 1} - \frac{1}{k-1}\Big)\,.$$
Thus in order to prove $y_1'(\theta_k) < 0$, we just need to show $\tfrac{e^{\lambda_k} - 1}{\lambda_k} > k-1$. Now notice that
$$f_1(\lambda_k) = \frac{\lambda_k(e^{\lambda_k - 1})}{e^{\lambda_1} - \lambda_k - 1} = k > k-1\,.$$
So it suffices to show $e^{\lambda_k} \geq 1 + \lambda_k + \lambda_k^2$. But for $x > 0$, 
$$e^x - 1 - x - x^2 > -\frac{-x^2}{2} + \frac{x^3}{6} + \frac{x^4}{24} = \frac{x^2}{24}(x^2 + 4x - 12) >0$$ 
whenever $x\geq2$. Since $f_1(\lambda_k) = k \geq 3$, $f_1(2) < 3$ and $f_1(x)$ is strictly increasing on $[0,\infty)$, it follows that $\lambda_k > 2$. On the other hand 
$$y_2'(\theta_k) = -(k-1)\mathfrak{p}_1(\overrightarrow{y}(\theta_k,\rho_k),\theta_k) =  -(k - 1)\frac{(1 - \theta_k)\lambda_k^2}{k(1-\theta_k)(e^\lambda_k - 1 - \lambda_k)} < 0\,.$$

\item Part (b) of Proposition~\ref{prop:4.2} already settles the $\theta \leq \theta_k$ part. For $\theta \geq \theta_k$ part, notice that from Lemma~\ref{lem:4.1}, Proposition~\ref{prop:4.2} and the previous parts of the current proposition we get an $\epsilon \in (0, 1 - \theta_k)$ such that 
$$F_1\big(y_1(\theta),\theta\big) < 0\,, \mbox{ and } k - \frac{1}{4} < \frac{k(1-\theta)}{y_2(\theta)} < k + \frac{1}{4}$$
on $(\theta_k,\theta_k+\epsilon]$. The former implies $y_1(\theta) < 0$ on $(\theta_k,\theta_k+\epsilon]$ and consequently 
$\mathfrak{p}_0\big(\overrightarrow{y}(\theta),\theta\big) = 0$. Hence for $\theta \in [\theta_k,\theta_k + \epsilon]$ we have from \eqref{eq:2.2} and \eqref{eq:2.3} that 
$$1 - \mathfrak{p}_1\big(\overrightarrow{y}(\theta),\theta\big) = \mathfrak{p}_2\big(\overrightarrow{y}(\theta),\theta\big) = f_1^{-1}\Big(\frac{k(1-\theta)}{y_2(\theta)}\Big)\frac{y_2(\theta)}{k(1-\theta)}\,.$$
Since $F_2\big(y_2(\theta),\theta\big)$ is Lipschitz continuous on $[\theta_k,\theta_k + \epsilon]$ by Lemma~\ref{lem:4.1}, it suffices to prove that $f_1^{-1}\big(\tfrac{k(1-\theta)}{y_2(\theta)}\big)$ is continuously differentiable on $[\theta_k,\theta_k+\epsilon]$. But we have already shown $\tfrac{k(1-\theta)}{y_2(\theta)}$ 
lies between $k-\tfrac{1}{4}$ and $k + \tfrac{1}{4}$ on the same set. The proof is now complete with the observation that $f_1^{-1}$ is continuously differentiable on $[k-\frac{1}{4},k+\frac{1}{4}]$. \qedhere
\end{enumerate}
\end{enumerate}
\end{proof}
Consider a new sequence of vectors defined as
\begin{equation}
\label{eq:4.7}
\overrightarrow{y}^*(\tau + 1) = \overrightarrow{y}^*(\tau) + n^{-1}\widetilde{\mathbb{A}}_{\tau}(\overrightarrow{y}^*(\tau) - \overrightarrow{y}(\tau/n)) + n^{-1}\overrightarrow{F}(\overrightarrow{y}(\tau/n), \frac{\tau}{n})\,,
\end{equation}
where $\overrightarrow{y}^*(0) \equiv \overrightarrow{y}(0,\rho)$ and $\widetilde{\mathbb{A}}_{\tau} \equiv \mathbb{I}_{\tau < \tau_n}\mathbb{A}(\overrightarrow{y}(\tau/n,\rho),\tau/n)$. Also define the positive definite matrices
\begin{equation}
\label{eq:4.8}
\mathbb{Q}_{\tau} = \widetilde{\mathbb{B}}_0^{\tau - 1}\mathbb{Q}(0,\rho)(\widetilde{\mathbb{B}}_0^{\tau - 1})^T + \frac{1}{n}\sum\limits_{\sigma = 0}^{\tau-1}\widetilde{\mathbb{B}}_{\sigma+1}^{\tau - 1}\mathbb{G}\big(\overrightarrow{y}(\sigma/n),\sigma/n\big)(\widetilde{\mathbb{B}}_{\sigma+1}^{\tau - 1})^T
\end{equation}
where $\widetilde{\mathbb{B}}_{\sigma}^{\tau}$ is same as in \eqref{eq:2.13}. \eqref{eq:4.7} and \eqref{eq:4.8} are discrete recursions corresponding to the mean and covariance of the process $\overrightarrow{z}'(.)$ of \eqref{eq:2.12}. The lemma below shows that $\overrightarrow y^*(\tau)$ and $\mathbb Q_\tau$ are near the solutions of appropriate ODE's up to time $\tau_n \equiv \lfloor n\theta_k - n^{\beta}\rfloor$ when $\rho$ is near 
$\rho_k$. The lemma and its proof are very similar to \cite[Lemma~4.3]{DM08} except for a few modifications.
\begin{lemma}
\label{lem:4.3}
Fixing $\beta \in (\frac{1}{2},1)$ and $\beta' < \beta$, we have for all sufficiently large $n$ and $|\rho - \rho_k|\leq n^{\beta' - 1}$,
\begin{equation}
\label{eq:4.9a}
\bigl\lvert y_1^*(\tau_n) + \frac{\partial{y_1}}{\partial{\theta}}(\theta_k,\rho_k)n^{\beta - 1} - (\rho - \rho_k)\frac{\partial{y_1}}{\partial{\rho}}(\theta_k,\rho_k)\bigr\rvert \leq Cn^{2(\beta-1)},
\end{equation}
\begin{equation}
\label{eq:4.9b}
\bigl\lvert y_2^*(\tau_n) - (1 - \theta_k) + \frac{\partial{y_2}}{\partial{\theta}}(\theta_k,\rho_k)n^{\beta - 1} - (\rho - \rho_k)\frac{\partial{y_2}}{\partial{\rho}}(\theta_k,\rho_k)\bigr\rvert \leq Cn^{2(\beta-1)},
\end{equation}
for some positive $C = C(\beta,\beta')$.\\
Furthermore, the matrices $\{\widetilde{\mathbb{B}}_{\sigma}^{\tau}: \sigma, \tau \leq n\}$ and their inverses are uniformly bounded with respect to the $L_2$-operator norm (denoted $\parallel.\parallel$) and
\begin{equation}
\label{eq:4.10}
\parallel\mathbb{Q}_{\tau_n} - \mathbb{Q}(\theta_k,\rho_k)\parallel \leq Cn^{\beta-1}
\end{equation}
for all $n$.
\end{lemma}
\begin{proof}
\label{proof4.3}
From part (a) of Proposition~\ref{prop:4.2} we get that $\overrightarrow{y}(\theta,\rho)\in\widehat{q}(\epsilon)$ for $\theta \leq 1 - 2\epsilon$ and $\rho \in [\epsilon, 1/\epsilon]$. 
Now choose $\epsilon < (1 - \theta_k)/2$. Part (b) of Proposition~\ref{prop:coro_crit} says $y_1(\theta, \rho_k) > 0$ for $0 \leq \theta < \theta_k$ while $y_1(\theta_k, \rho_k) = 0$. From the same proposition we also have that $\frac{\partial{y_1}}{\partial{\theta}}(\theta_k,\rho_k) 
< 0$. Hence from Lipschitz continuity of $\mathfrak{p}_a(\overrightarrow{x},\theta)$, ($a=0,1,2$) on $\widehat{q}(\epsilon)$, as given by Lemma~\ref{lem:4.1}, we get $y_1(\theta,\rho_k) \geq c_0n^{\beta - 1}$ for some positive $c_0 = c_0(\epsilon)$ 
whenever $0\leq\theta\leq\theta_n = \tau_n/n$. Combined with the fact that $(\theta,\rho)\shortmid\rightarrow \overrightarrow{y}$ is Lipschitz continuous on $(\theta,\rho)\in[0, 1 -\epsilon)\times[\epsilon, 1/\epsilon]$ (part~(a), Proposition~\ref{prop:4.2}), this further implies $y_1(\theta,\rho)\geq c_1n^{\beta - 1}$ for some positive $c_1 = c_1(\epsilon,\beta,\beta')$ whenever $0\leq\theta\leq\theta_n$, $|\rho - \rho_k| \leq n^{\beta' - 1}$ and $n\geq n_0(\epsilon,\beta,\beta')$. 
Consequently $\overrightarrow{y}\in\widehat{q}_+(\epsilon)$ for $0\leq\theta\leq\theta_n$, 
$|\rho - \rho_k|\leq n^{\beta'-1}$ and all such $n$. Hence by Lemma~\ref{lem:4.1} we obtain that the entries of the matrices $\mathbb{A}_{\tau}$ are bounded uniformly in $\tau \leq \tau_n$, $|\rho - \rho_k|\leq n^{\beta'-1}$ and $n$ as before. From the expression of $\mathbb{B}_{\sigma}^{\tau}$ in \eqref{eq:2.13} we thus conclude that the matrices $\{\mathbb{B}_{\sigma}^{\tau}:\sigma, \tau \leq n\}$ and their inverses are bounded w.r.t. the $L_2$ operator norm uniformly in 
$|\rho - \rho_k|\leq n^{\beta'-1}$ and $n$.
The remaining part of the proof can be completed by mimicking the same in \cite{DM08}.
\end{proof}
\section{Gaussian approximation and the proof of Proposition~\ref{prop:prep}}
\label{sec:approximation}
This section is devoted to proving 
Proposition~\ref{prop:prep}. Notice that $n_{core} = n - \tau_c$ where $\tau_c$ is the first time the process 
$\{\overrightarrow{z}(\tau)\}$ hits the $z_1 = 0$ line and $m_{core} = z_2(\tau_c)$. Our goal is to estimate the 
distribution of $z_2(\tau_c) - (n-\tau_c)$. We will begin with an approximation of the process $\overrightarrow{z}'(\tau)$ of \eqref{eq:2.12} by a 
suitable Gaussian process. To this end let $\mathscr G_d(.|\overrightarrow{x},\mathbb{A})$ denote the $d$-dimensional Gaussian distribution with mean 
$\overrightarrow{x}$ and covariance $\mathbb{A}$. We denote by $\overline{\mathbb{P}}^G_{n,\rho}(.)$ the law of a $\mathbb{R}^2$-valued process
$\{\overrightarrow{z}^{''}(\tau)\}_{0\leq\tau\leq n}$ where $\overrightarrow{z}^{''}(0)$ has the uniform distribution $\mathbb{P}_{\mathcal{G}_k(n, \lfloor n\rho \rfloor)}$ on the graph ensemble $\mathcal{G}_k(n, \lfloor n\rho \rfloor)$ and
\begin{eqnarray}
\label{eq:Gauss1}
&&\overline{\mathbb{P}}^G_{n,\rho}(\overrightarrow{z}^{''}(\tau + 1)=\overrightarrow{z}^{''}(\tau) + \Delta'_{\tau} + \widetilde{\mathbb{A}}_{\tau}(n^{-1}\overrightarrow{z}^{''} - \overrightarrow{y}(\tau/n))|\overrightarrow{z}^{''}(\tau) = \overrightarrow{z}^{''})\nonumber \\
&=&\mathscr G_2\Big(\Delta'_{\tau}|\overrightarrow{F}(\overrightarrow{y}(\tau/n),\tau/n),\mathbb{G}(\overrightarrow{y}(\tau/n),\tau/n)\Big),
\end{eqnarray}
for $\tau < \tau_n$. For $\tau \in J_n \equiv [n\theta_k - n^\beta, n\theta_k + n^\beta]$ 
and $|\rho - \rho_k|\leq n^{\beta' -1}$, we expect $\Delta'_{\tau}$'s to be distributed like $\Delta'_{\lfloor n\theta_k\rfloor}$ which has mean $\big((k-1)\mathfrak{p}_1(\overrightarrow{y}(\theta_k,\rho_k),\theta_k) - 1,-(k-1)\mathfrak{p}_1(\overrightarrow{y}(\theta_k,\rho_k),\theta_k)\big)$ and a positive semidefinite covariance matrix $\mathbb{G}'(\theta_k,\rho_k)$ whose both diagonal entries are $\mathbb{G}_{11}(\overrightarrow{y}(\theta_k,\rho_k),\theta_k)$. So for $\tau \geq \tau_n$, we modify the transition kernel as follows:
\begin{eqnarray}
\label{eq:Gauss2}
&&\overline{\mathbb{P}}^G_{n,\rho}(\overrightarrow{z}^{''}(\tau + 1)=\overrightarrow{z}^{''}(\tau) + \Delta'_{\tau} + \widetilde{\mathbb{A}}_{\tau}(n^{-1}\overrightarrow{z}^{''} - \overrightarrow{y}(\tau/n))|\overrightarrow{z}^{''}(\tau) = \overrightarrow{z}^{''})\nonumber\\
&&=\mathscr G_2\Big(\Delta'_{\tau}|\overrightarrow{F}(\overrightarrow{y}(\theta_k,\rho_k),\theta_k),\mathbb{G}'(\theta_k,\rho_k)\Big).
\end{eqnarray}

\begin{lemma}
\label{lem:6.1}
Fixing $\beta\in(\frac{1}{2},\frac{3}{4})$, set $0 < 
\beta' < \beta$ and $2\beta - 1 < \beta^{''} < 2(1 - 
\beta)$. Then there exist $\alpha = \alpha(\beta, \beta'', k) > 0$, $0 < \delta_1 = \delta_1(\beta, \beta'') < \frac{1}{2}$ and a coupling between the processes $\{\overrightarrow{z}'(\tau)\}$ of distribution $\overline{\mathbb{P}}_{n,\rho}(.)$ and $\{\overrightarrow{z}^{''}(\tau)\}$ of distribution $\overline{\mathbb{P}}^G_{n,\rho}(.)$ such that for all large $n$ and $|\rho - \rho_k|\leq n^{\beta' - 1}$,
\begin{equation}
\label{eq:5.18_2}
\mathbb{P}(\sup\limits_{\tau \in J_n}\parallel\overrightarrow{z}'(\tau) - \overrightarrow{z}^{''}(\tau)\parallel\geq n^{\delta_1})\leq \alpha n^{-\beta^{''}}\,.
\end{equation}
\end{lemma}
\begin{proof}
The key to this lemma is the fact that under certain conditions we can couple of partial sums of independent random vectors with the same for independent 
gaussian random vectors having similar moments. For the $\tau < \tau_n$ part we recall from \eqref{eq:2.12} that $\overrightarrow z'(\tau_n) - \overrightarrow z'(0)$ can 
be written as a sum of independent random vectors. Then by a multidimensional version of a strong approximation result of Sakhanenko (see \cite[Theorem~1.2]{zaitsevI}), there exist a sequence of independent gaussian vectors $\Delta'_{\tau}$'s distributed like in \eqref{eq:Gauss1}, $c_2 = c_2(k) > 0$ and $\alpha_0 = \alpha_0(k)$ such that:
\begin{equation}
\label{eq:strong_approximate1}
\mathbb{P}(\parallel \sum_{\tau < \tau_n}(\Delta_\tau - \Delta_\tau')\parallel \geq c_2\log n) \leq \alpha_0 n^{-1}
\end{equation}
for $|\rho - \rho_k|\leq n^{\beta' - 1}$ and all $n$. Furthermore $\Delta_\tau'$'s are independent with 
$\overrightarrow{z}'(0)$. \cite[Theorem~1.2]{zaitsevI} requires some ellipticity conditions on the covariance matrices of the summands which follow in this case from the (easily verifiable) fact that $\mathfrak{p}_a(\overrightarrow{y}(\theta,\rho),\theta)$'s are uniformly bounded away from 0 when $0\leq\theta\leq\theta_k/2$ and $|\rho - \rho_k|\leq n^{\beta'-1}$.

For the $\tau \geq \tau_n$ part, first notice that $$\var(\Delta_{\tau, 1} + \Delta_{\tau, 2}) = (k-1)\mathfrak{p}_0(\tau/n,\rho)(1 - \mathfrak{p}_0(\tau/n,\rho))\,,$$
where $\Delta_{\tau} = 
(\Delta_{\tau, 1},\Delta_{\tau, 2})$. 
From part~(a), Proposition~\ref{prop:4.2} and Lemma~\ref{lem:4.1}, we get for $|\rho - \rho_k| \leq n^{\beta' - 1}$, $\tau\in J_n$ and all large $n$,
\begin{eqnarray*}
\mathfrak{p}_0(\tau/n,\rho) = \mathfrak{p}_0(\tau/n,\rho) - \mathfrak{p}_0(\theta_k,\rho_k) \leq C_6(|\theta_k - \tau/n| + |\rho - \rho_k|) \leq 2C_6n^{\beta-1}\,,
\end{eqnarray*}
where $C$ is a positive constant. Hence
$$\sum\limits_{\tau \in J_n}\var(\Delta_{\tau, 1} + \Delta_{\tau, 2}) \leq 2(k - 1)C_6n^{2\beta - 1}\,.$$
Now by Kolmogorov's maximal inequality
\begin{eqnarray}
\label{eq:strong_approximate2}
\mathbb{P}\bigg(\max_{\tau\in J_n}|\sum_{\tau_n \leq t < \tau}\big((\Delta_{t, 1} + \Delta_{t, 2}) - \mathbb{E}(\Delta_{t, 1} + \Delta_{t, 2})\big)|
\geq n^{\delta}\bigg) \leq 2(k-1)C_6n^{2\beta - 1 - 2\delta}\,.
\end{eqnarray}
So choosing $\delta_1 = \beta + \frac{\beta^{''}}{2} - \frac{1}{2}$ we incur an error of at most $n^{\delta_1}$ with probability at least $1 - 2C_6(k - 1)n^{-\beta^{''}}$ if we replace $\Delta_{\tau, 2}$ with $\mathbb{E}(\Delta_{\tau, 1} + \Delta_{\tau, 2}) - 
\Delta_{\tau, 1}$. Since $\Delta_{\tau_n + i, 1}$'s are independent and uniformly bounded by $4k$, we can apply Sakhanenko's refinement of the Hungarian construction (see \cite{sakhanenko1984, Shao95}) to deduce the existence of a sequence of independent gaussian variables $\{\Delta''_{\tau, 1}\}_{\tau \geq \tau_n}$'s such that 
$$\E {\Delta_{\tau, 1}''} = \E \Delta_{\tau, 1} = F_1(\overrightarrow{y}(\tau/n,\rho),\tau/n)\,, \mbox{ }\var {\Delta''_{\tau, 1}} = \var \Delta_{\tau, 1} = \mathbb{G}_{11}(\overrightarrow{y}(\tau/n,\rho),\tau/n)\,$$
and 
\begin{equation}
\label{eq:strong_approximate3}
\P \big(\sup_{\tau \in J_n} \big \lvert\sum_{\tau_n \leq t < \tau}({\Delta_{t, 1}''} - \Delta_{t, 1})\big \rvert
 \geq c_3\log n\big) \leq \alpha_1n^{-1}
\end{equation}
for some $c_3 = c_3(\beta, k) > 0$ and $\alpha_1 = \alpha_1(\beta, k)$.
$${\Delta_{\tau, 1}'} = \frac{\sqrt{\mathbb{G}_{11}(\overrightarrow{y}(\theta_k,\rho_k),\theta_k)}}{\sqrt{\mathbb{G}_{11}(\overrightarrow{y}(\tau/n,\rho),\tau/n)}}\big({\Delta_{\tau, 1}''} - F_1(\overrightarrow{y}(\tau/n,\rho), \tau/n)\big) + F_1(\overrightarrow{y}(\theta_k, \rho_k), \theta_k)\,,$$
and 
$${\Delta_{\tau, 2}'} = \E(\Delta_{\tau, 1} + \Delta_{\tau, 2}) - {\Delta_{\tau, 1}'}\,.$$
We claim that the process $\overrightarrow z''(\tau)$ with increments $\Delta_\tau' = (\Delta_{\tau, 1}', {\Delta_{\tau, 2}'})$ satisfy \eqref{eq:5.18_2} for 
$\delta_1, \beta''$ and a suitable choice of $\alpha$. We will justify this claim in two steps. First notice that, by Lipschitz continuity of $\mathbb G_{11}(., .)$, there exists a positive constant $C_7$ such that
\begin{eqnarray*}
\bigl\lvert\mathbb{G}_{11}(\overrightarrow{y}(\tau/n,\rho),\tau/n) - \mathbb{G}_{11}(\overrightarrow{y}(\theta_k,\rho_k),\theta_k)\bigr\rvert \leq C_7n^{\beta - 1}\,,
\end{eqnarray*}
for all $\tau \in J_n$ and $|\rho - \rho_k|\leq n^{\beta' - 1}$.
Writing
\begin{eqnarray*}
\Delta''_{\tau,1} - \E \Delta''_{\tau,1} &=& \frac{\sqrt{\mathbb{G}_{11}(\overrightarrow{y}(\theta_k,\rho_k),\theta_k)}}{\sqrt{\mathbb{G}_{11}(\overrightarrow{y}(\tau/n,\rho),\tau/n)}}(\Delta''_{\tau,1}-
\mathbb{E}\Delta''_{\tau,1}) + \Big(1-\frac{\sqrt{\mathbb{G}_{11}(\overrightarrow{y}(\theta_k,\rho_k),\theta)}}{\sqrt{\mathbb{G}_{11}(\overrightarrow{y}(\tau/n,\rho),\tau/n)}}\Big)
(\Delta''_{\tau,1} - \mathbb{E}\Delta''_{\tau,1}) \\ 
&=& \Delta'''_{\tau,1} + \Delta''''_{\tau,1}
\end{eqnarray*}
we notice that $\Delta'''_{\tau,1}$'s are independent Gaussian variables with mean 0 and variance $\leq C_8
n^{2\beta - 2}$ for some positive constant $C_8$. 
Hence by a similar application of Kolmogorov's inequality as in \eqref{eq:strong_approximate2} we get
\begin{equation}
\label{eq:strong_approximate4}
\mathbb{P}\big(\max_{\tau\in J_n}|\sum_{\tau_n \leq t < \tau} \Delta_{\tau, 1}''''|
\geq n^{\delta_1}\big) \leq 2(k-1)C_8n^{3\beta - 2 - 2\delta_1} \leq 2(k-1)C_8n^{-\beta''}\,.
\end{equation}
In view of \eqref{eq:strong_approximate4} and the definition of $\Delta_\tau'$, it only remains to bound the following:
$$\max_{\tau \in J_n}\big|\sum_{\tau_n \leq t < \tau}\big(F_1(\overrightarrow y(t/n, \rho), t/n) - F_1(\overrightarrow y(\theta_k, \rho_k), \theta_k)\big)\big|\,.$$
We can do this by applying Lemma~\ref{lem:4.1} and Proposition~\ref{prop:4.2}, part~(a), which gives us
\begin{equation}
\label{eq:strong_approximate5}
\max_{\tau \in J_n}\big|\sum_{\tau_n \leq t < \tau}\big(F_1(\overrightarrow y(t/n, \rho), t/n) - F_1(\overrightarrow y(\theta_k, \rho_k), \theta_k)\big)\big| \leq C_9n^{2\beta - 1} \leq C_9n^{\delta_1}\,,
\end{equation}
for $|\rho - \rho_k|\leq n^{\beta' - 1}$ and some positive constant $C_9$. \eqref{eq:5.18_2} now follows from combining the displays \eqref{eq:strong_approximate1}, \eqref{eq:strong_approximate2}, \eqref{eq:strong_approximate3}, \eqref{eq:strong_approximate4} and \eqref{eq:strong_approximate5}.
\end{proof}
We are interested in the behavior of $\{(\overrightarrow{z}^{''}(\tau), \tau)\}_{\tau \geq 
\tau_n}$. After a little reshuffling we get the following decent expression:
\begin{equation}
\label{eq:expr1}
\overrightarrow{z}^{''}(\tau) =   \overrightarrow{\xi}^*+ \big(S_{\tau},-S_{\tau}\big) + \big(\frac{\partial y_1}{\partial \theta}(\tau_n - n\theta_k),\frac{\partial y_2}{\partial \theta}(\tau_n - n\theta_k)\big) + \big(0,n(1-\theta_k)\big)
\end{equation}
where $S_{\tau} = \sum\limits_{\tau_n}^{\tau - 1}\Delta'_{t,1}$ and the process $\{S_{\tau}\}_{\tau \geq \tau_n}$ is independent of $\overrightarrow{\xi}^*$ which is defined as
$$\overrightarrow{\xi}^* = \overrightarrow{z}^{''}(\tau_n) - \big(\frac{\partial y_1}{\partial \theta}(\tau_n - n\theta_k),\frac{\partial y_2}{\partial \theta}(\tau_n - n\theta_k)\big) - \big(0,n(1-\theta_k)\big)\,.$$
Since we are dealing with first hitting times it is convenient to consider processes defined on $\R^+$. Also the presence of an implicit upper bound $n$ on $\tau$ is 
inconvenient. To get around these issues we first extend the definition of the processes $\{\overrightarrow{z}(\tau)\}$ and 
$\{\overrightarrow{z}^{''}(\tau)\}$ beyond $n$. For the process $\{\overrightarrow{z}(\tau)\}$ this is achieved by setting $\overrightarrow{z}(\tau)=0$ for $\tau \geq n$ and for $\{\overrightarrow{z}^{''}(\tau)\}$ we simply 
carry on \eqref{eq:Gauss2} for all $n$. Now we extend these two processes to all $t \in \R^+$ by linear 
interpolation. This minor change has some other technical 
advantages as well. Both the processes now almost surely 
hit the $x = 0$ line at finite time. Furthermore the minimum $\tau$ such that $z_1(\tau)=0$ is still $\tau_c$. 
Define $\tau^G_c$ for $\{\overrightarrow{z}^{''}\}$ as the first time $t \geq \tau_n$ such that $z^{''}(\tau) = 
0$ if $z^{''}(\tau_n) \geq 0$ and $\tau_n$ otherwise. The following lemma is an important intermediate step for our proof of Lemma~\ref{lem:6.2} which shows that $(z_2(\tau_c), \tau_c)$ and $(z_2(\tau_c^G), \tau_c^G)$ 
are pretty close. This lemma is similar in essence to \cite[Corollary~5.3]{DM08}.
\begin{lemma}
\label{cor:5.3}
Fix $\beta \in (\tfrac{1}{2},1)$ and $0 < \beta' < \beta$. 
Then there exist $C = C(\beta, \beta') > 0$ and $\eta = \eta(\beta) > 0$ such that
$$\widehat{\mathbb{P}}_{n,\rho}\big(\{\min\limits_{\tau\in[0,n\theta_k - n^\beta]}z_1(\tau)\leq n^{\beta'}\}\cup\{z_1(\lfloor n\theta_k+n^\beta\rfloor)\geq -n^{\beta'}\}\big) \leq Ce^{-n^{\eta}}\,,$$
all $n$ and $|\rho - \rho_k|\leq n^{\beta'-1}$.
\end{lemma}
\begin{proof}
We showed in the proof of Lemma~\ref{lem:4.3} that $y_1(\theta, \rho) \geq c_1 n^{\beta-1}$ for all $0\leq \theta \leq \theta_n$, $|\rho - \rho_k|\leq n^{\beta' - 
1}$ and large $n$. By part (b) of Lemma~\ref{lem:concentration} 
we then get that for some $C' = C'(\beta,\beta') > 0$ and $n$ large enough
\begin{equation*}
\mathbb{E}z_1(\tau) \geq ny_1(\tau/n,\rho) - C_1\sqrt{n\log{n}} \geq c_1 n^{\beta} - C_1\sqrt{n\log{n}}\geq C'n^{\beta}\,,
\end{equation*}
whenever $\tau\in[0,n\theta_k-n^\beta]$ and $|\rho - \rho_k|\leq n^{\beta' - 1}$. 
Now applying part~(a) of Lemma~\ref{lem:concentration}, we get that for any $\eta < (2\beta - 1)/2$, some $C''=C''(\beta, \beta', \eta) > 0$ and $n$ large enough
$$\widehat{\mathbb{P}}_{n,\rho}(z_1(\tau)\leq n^{\beta'})\leq\widehat{\mathbb{P}}_{n,\rho}(||\overrightarrow{z}(\tau) - \mathbb{E}\overrightarrow{z}(\tau)||\geq C'n^{\beta}/2)\leq C''e^{-n^{2\eta}},$$
whenever $\tau \in [0,n\theta_k-n^\beta]$ and $|\rho - \rho_k|\leq n^{\beta' - 1}$. 
Similarly we can derive a similar exponential bound on $\widehat{\mathbb{P}}_{n,\rho}\{z_1(\lfloor n\theta_k+n^\beta\rfloor)\geq -n^{\beta'}\}$. The result now follows by applying a union bound over $\tau \in [0,n\theta_k-n^\beta] \cup \{\lfloor n\theta_k+n^\beta\rfloor\}$.
\end{proof}
\begin{lemma}
\label{lem:6.2}
Fix $\beta \in (\frac{1}{2},\frac{3}{4})$ and 
$0 < \beta' < \beta$. There exist $0 < \delta_2 < \frac{1}{2}$ and $\delta_3 > 0$ depending only on $\beta, \beta'$, $B = B(\beta, \beta', k)> 0$ and a coupling of the processes $\{\overrightarrow{z}(\tau)\}_{\tau \in \mathbb{R}^+}$ and $\{\overrightarrow{z}^{''}(\tau)\}_{\tau \in \mathbb{R}^+}$, such that for all $n$ and $|\rho - \rho_k|\leq n^{\beta' - 1}$,
\begin{equation*}
\P\big(z^{''}_1(\tau^G_c) = 0, |\tau_c - \tau^G_c| + |z_2(\tau_c) - z_2^{''}(\tau^G_c)|\leq n^{\delta_2}\big) \geq  1 - Bn^{-\delta_3}\,.
\end{equation*}
\end{lemma}
\begin{proof}
Since $z_1^{''}(\tau) - z_1^{''}(\tau_n)$ is a sum of i.i.d random variables with negative mean (see Proposition~\ref{prop:coro_crit}), it is clear 
that $\tau^G_c$ is almost surely finite. Thus
$$\P(z_1^{''}(\tau_n) \geq 0, z_1^{''}(\tau^G_c) \neq 0) = 0\,.$$
Now we couple the processes $\{\overrightarrow{z}(\tau)\}_{\tau \geq 0}$ and $\{\overrightarrow{z}^{''}(\tau)\}_{\tau \geq 0}$ by joining the couplings of Lemma~\ref{lem:6.1}, Proposition~\ref{prop:5.5} and Corollary~\ref{cor:5.5}. 
We can extend this coupling to all $t \in \R^+$ by linear 
interpolation. From the same results we then get the numbers $0 < \delta_2' = \delta_2'(\beta, \beta') < \frac{1}{2}$, $\delta_3' = \delta_3'(\beta, \beta') > 0$ and $B' = B'(\beta, \beta', k)$, such that
\begin{equation}
\label{eq:time_bound1}
\P \big(\sup\limits_{\tau \in J_n}\parallel\overrightarrow{z}(\tau) - \overrightarrow{z}^{''}(\tau)\parallel \geq n^{\delta_2'}\big) \leq B'n^{-\delta_3'}\,,
\end{equation}
for all $n$ and $|\rho - \rho_k|\leq n^{\beta' - 1}$. \eqref{eq:time_bound1} combined with 
Corollary~\ref{cor:5.5} and Lemma~\ref{cor:5.3} gives us
\begin{equation}
\label{eq:time_bound2}
\P \big(z_1''(\tau_n) \geq 0, \tau_c \in J_n, \tau_c^G \in J_n\big) \geq 1 - B''n^{-\delta_3'}\,,
\end{equation}
for all $n$, $|\rho - \rho_k|\leq n^{\beta' - 1}$ and 
some $B'' = B''(\beta, \beta', k)$. We still need one 
more ingredient. Since $\Delta_\tau''$'s are i.i.d. Gaussian variables, it follows by Dudley's entropy bound on the supremum of a Gaussian process (see, e.g.,  \cite[Theorem~4.1]{A90}) and Gaussian concentration inequality (see e.g., \cite[Equation (7.4), 
Theorem~7.1]{L01}) that
\begin{equation}
\label{eq:time_bound3}
\P \Big(\max_{\tau \in J_n}\big(z_1''(\tau_{\delta_2, +}) - z_1''(\tau) - \E (z_1''(\tau_{\delta_2, +}) - z_1''(\tau))\big) \geq n^{3\delta_2/4}\Big) \leq B'''n^{-1}\,,
\end{equation}
for some $B''' = B'''(\beta, \beta', k)$ and $\tau_{\delta_2, +} = (\tau + n^{\delta_2}) \wedge 
(n\theta_k + n^{\beta})$. But 
$$\E (z_1''(\tau + n^{\delta_2}) - z_1''(\tau)) = n^{\delta_2}F_1(\overrightarrow y(\theta_k ,\rho_k), \theta_k)\,,$$
where $F_1(\overrightarrow y(\theta_k ,\rho_k), \theta_k) 
< 0$. The lemma now follows from this fact together with \eqref{eq:time_bound1}, \eqref{eq:time_bound2} and \eqref{eq:time_bound3}.
\end{proof}
Now let us revisit \eqref{eq:expr1}. A routine algebra yields that when $z_1^{''}(\tau^G_c) = 0$, we have
\begin{equation}
\label{eq:conv_expr}
z_2^{''}(\tau^G_c) - (n - \tau^G_c) = (\xi^*_1 + \xi^*_2)\,,
\end{equation}
where $\overrightarrow{\xi}^* = (\xi_1^*, \xi_2^*)$. The following is an immediate consequence of Lemma~\ref{lem:6.2} and \eqref{eq:conv_expr}.
\begin{lemma}
\label{lem:6.3}
Fix $\beta \in (\frac{1}{2},\frac{3}{4})$ and 
$\frac{1}{2} < \beta' < \beta$. There exist $0 < \delta_2 < \frac{1}{2}$ and $\delta_3 > 0$ depending only on $\beta, \beta'$, $B_1 = B_1(\beta, \beta', k) > 0$ such that for all $n$, $A > 0$ and $|\rho - \rho_k|\leq n^{\beta' - 1}$,
\begin{eqnarray*}
\overline{\mathbb{P}}^G_{n,\rho}\big((\xi^*_1 + \xi^*_2) \geq A\log n + n^{\delta_2}\big) - B_1n^{-\delta_3} &\leq& \mathbb{P}_{n,\rho}\big(z_2(\tau_c) - (n - \tau_c) \geq A\log n\big)\\ &\leq&  \overline{\mathbb{P}}^G_{n,\rho}\big((\xi^*_1 + \xi^*_2) \geq A\log n - n^{\delta_2}\big) + B_1n^{-\delta_3}\,,
\end{eqnarray*}
and
\begin{eqnarray*}
\overline{\mathbb{P}}^G_{n,\rho}\big((\xi^*_1 + \xi^*_2) \leq -A\log n - n^{\delta_2}\big) - B_1n^{-\delta_3}&\leq& \mathbb{P}_{n,\rho}\big(z_2(\tau_c) - (n - \tau_c) \leq -A\log n\big)\\ &\leq&  \overline{\mathbb{P}}^G_{n,\rho}\big((\xi^*_1 + \xi^*_2) \leq -A\log n + n^{\delta_2}\big) + B_1n^{-\delta_3}\,.
\end{eqnarray*}
\end{lemma}
\emph{Proof of Proposition~\ref{prop:prep}:} Notice that,
$$\overrightarrow{\xi}^* = \big(\overrightarrow{z}^{''}(\tau_n) - \widetilde{\mathbb{B}}_0^{\tau_n - 1}\overrightarrow{z}(0)\big) - \big(\frac{\partial y_1}{\partial \theta}(\tau_n - n\theta_k),\frac{\partial y_2}{\partial \theta}(\tau_n - n\theta_k)\big) - \big(0,n(1-\theta_k)\big) + \widetilde{\mathbb{B}}_0^{\tau_n - 1}\overrightarrow{z}(0)\,.$$
From \eqref{eq:2.14} and the coupling defined 
in Lemma~\ref{lem:6.1} we can see that $\overrightarrow{z}^{''}(\tau_n) - \widetilde{\mathbb{B}}_0^{\tau_n - 1}\overrightarrow{z}(0)$ and $\overrightarrow{z}(0)$ are 
independent. Also from \eqref{eq:Gauss1}, \eqref{eq:4.7} and \eqref{eq:4.8} we have
$$\overrightarrow{z}^{''}(\tau_n) - \widetilde{\mathbb{B}}_{0}^{\tau_n - 1}\overrightarrow{z}(0) \myeq \mathscr G_2\Big(. | n\big(\overrightarrow{y}^*(\tau_n) - \widetilde{\mathbb{B}}_{0}^{\tau_n - 1}\overrightarrow{y}(0,\rho)\big), n\big(\mathbb{Q}_{\tau_n} - \widetilde{\mathbb{B}}_{0}^{\tau_n - 1}\mathbb{Q}_{\tau_n}(\widetilde{\mathbb{B}}_{0}^{\tau_n - 1})^T\big)\Big)\,.$$
Hence by Lemma~\ref{lem:normal_approx} we get for some $\kappa_3 > 0$,
\begin{eqnarray*}
\sup\limits_{x\in \mathbb{R}}\bigl\lvert\mathbb{P}(\xi^*_1 + \xi^*_2\leq x)-\int\limits_{z_1 + z_2\leq x}\phi_2(\overrightarrow{z}|n\overrightarrow{y}^*(\tau_n);n\mathbb{Q}(\tau_n))d\overrightarrow{z}\bigr\rvert \leq \kappa_3n^{-1/2}.
\end{eqnarray*}
On the other hand Lemma~\ref{lem:4.3} tells us
$$\parallel n\mathbb{Q}_{\tau_n} -n\mathbb{Q}_{\theta_k,\rho_k}\parallel \leq Cn^{\beta}\,,$$ and
\begin{eqnarray*}
&&\parallel n\overrightarrow{y}^*(\tau_n) - n\big((\rho - \rho_k)\frac{\partial{y_1}}{\partial{\rho}}(\theta_k,\rho_k),(\rho - \rho_k)\frac{\partial{y_2}}{\partial{\rho}}(\theta_k,\rho_k)\big) - \big(\frac{\partial y_1}{\partial \theta}(\tau_n - n\theta_k),\frac{\partial y_2}{\partial \theta}(\tau_n - n\theta_k)\big)\\
&& - \big(0,n(1-\theta_k)\big)\parallel \leq 2Cn^{2\beta - 1}
\end{eqnarray*}
for some $C = C(\beta, \beta')$. Now using formula for the Kullback-Leibler divergence between two multivariate Gaussian distributions and Pinsker's inequality (see, e.g., \cite[p.~132]{tsybakov2009})
we get for $\rho = \rho_k + rn^{-1/2}$,
\begin{equation}
\label{eq:fin_cut}
\parallel\mathcal{L}\big(\sqrt{n}(\xi_1(r) + \xi_2(r))\big) - \mathcal{L}(\xi^*_1 + \xi^*_2)\parallel_{TV} \leq \max(n^{(4\beta - 3)/2},n^{(\beta - 1)/2}),
\end{equation}
where $\overrightarrow{\xi}(r)$ is same as in 
Proposition~\ref{prop:prep} and $\mathcal L(X)$ denotes 
the law of the random variable $X$. The proof now follows from \eqref{eq:fin_cut} and Lemma~\ref{lem:6.3}.

\end{document}